\documentclass[11pt]{article}
\usepackage{amsmath}
\usepackage{amsfonts}
\usepackage{amsthm}
\usepackage{fullpage}
\usepackage{mathtools}
\usepackage{graphicx}
\usepackage[utf8]{inputenc}

\newtheorem{theorem}{Theorem}

\newtheorem{lemma}[theorem]{Lemma}

\newtheorem{conjecture}[theorem]{Conjecture}
\newtheorem{definition}[theorem]{Definition}

\newtheorem{remark}[theorem]{Remark}
\newcommand{\M}{\mathcal{M}}
\newcommand{\one}{\mathbf{1}}
\title{Exponential Lower Bounds on Spectrahedral Representations of Hyperbolicity Cones \thanks{This research was supported by NSF Grants CCF-1553751,
 NSF CCF-1343104 and NSF CCF-1407779, and a Sloan Research Fellowship for NS.}}
\author{Prasad Raghavendra\\ EECS\\ UC Berkeley \and Nick Ryder\\Mathematics\\UC Berkeley \and Nikhil Srivastava\\Mathematics\\UC Berkeley \and Benjamin Weitz\\Pinterest}
\begin{document}
\maketitle
\begin{abstract}
The Generalized Lax Conjecture asks whether every hyperbolicity cone is a
	section of a semidefinite cone of sufficiently high dimension. We prove
	that the space of hyperbolicity cones of hyperbolic polynomials of
	degree $d$ in $n$ variables contains $(n/d)^{\Omega(d)}$ pairwise
	distant cones in the Hausdorff metric, and therefore that any semidefinite
	representation of such polynomials must have dimension at least
	$(n/d)^{\Omega(d)}$ (even allowing a small approximation error). The cones are perturbations of the hyperbolicity cones of
	elementary symmetric polynomials. Our proof contains several ingredients of
	independent interest, including the identification of a large subspace
	in which the elementary symmetric polynomials lie in the relative
	interior of the set of hyperbolic polynomials, and a quantitative
	generalization of the fact that a real-rooted polynomial with two
	consecutive zero coefficients must have a high multiplicity root at
	zero.
\end{abstract}
\newcommand{\R}{\mathbb{R}}
\newcommand{\hdist}{\mathsf{hdist}}
\newcommand{\mdist}{\mathsf{mdist}}
\section{Introduction}

A homogeneous polynomial $p$ in $\R[x_1,\ldots,x_n]$ is said to be {\em hyperbolic} with respect to a direction $e \in \R^n$ if $p(e)\neq 0$ and all its univariate restrictions along the direction $e$ are real-rooted, i.e., for every $x \in \R^n$, the polynomial $p_x(t) = p(t e + x)$ has all real roots.
G\aa rding showed that every hyperbolic polynomial $p$ is associated with a closed\footnote{The usual definition considers open cones, but we will work with their closures instead.} convex cone $K_p$ defined as follows,
\[ K_p = \{x \in \R^n | \text{ all roots of } p(te + x) \text{ are non-positive} \}. \]
The cone $K_p$ is referred to as the hyperbolicity cone associated with the polynomial $p$.  

From the standpoint of convex optimization, hyperbolicity cones yield a rich
family of convex sets that one can efficiently optimize over --- in particular,
interior point methods can be used to efficiently optimize over the
hyperbolicity cone $K_p$ for a polynomial $p$, given an oracle to evaluate $p$
and its gradient and Hessian \cite{guler1997,renegar2006hyperbolic}. This
optimization primitive,  referred to as {\it hyperbolic programming}, captures
linear and semidefinite programming as special cases.  Specifically, the
hyperbolicity cone for the polynomial $p(x_1,\ldots,x_n) = \prod_{i = 1}^n x_i$
is the positive orthant $\R_+^n$, which corresponds to linear programming, while
semidefinite programming arises from the symmetric determinant polynomial
$\det(X)$. 
Is hyperbolic programming as an algorithmic primitive strictly more
 powerful than semidefinite programming?
Can hyperbolicity cones other than these two be harnessed towards obtaining
 better algorithms for combinatorial optimization problems?  These compelling questions remain
 open.  

Lately, the relationship between hyperbolicity cones and the semidefinite cone has been a subject
of intense study. It is easy to see that every section of the semidefinite cone is a hyperbolicity cone, so it is natural to ask
whether the converse holds, i.e., if every
hyperbolicity cone can be realized as a section of the semidefinite cone.
Formally, a {\it spectrahedral cone} $K$ is specified by:
$$ K = \{ x \in \R^n \mid \sum_i A_i x_i \succeq 0 \},$$ for some matrices $\{A_i\}_{i \in [n]} \in
\R^{B \times B}$.  It is conjectured that every hyperbolicity cone is a
spectahedral cone.

\begin{conjecture} (Generalized Lax Conjecture)
Every hyperbolicity cone is a spectahedral cone, i.e., a linear section of the cone of positive semidefinite matrices in some dimension.
\end{conjecture}

The Lax conjecture in its original stronger algebraic form asked whether every
polynomial in three variables hyperbolic with respect to the
direction $(1,0,0)$, could be written as $\det(xI+yB+zC)$ for some symmetric
matrices $B,C$ (sometimes called a {\em definite determinantal representation}). This immediately implies that all hyperbolicity cones in three
dimensions are spectrahedral, and was proved by Helton and Vinnikov and Lewis,
Parrilo, and Ramana \cite{helton2007linear,lewis2005lax}. The algebraic
conjecture is easily seen to be false for $n>3$ by a count of parameters, since
the set of hyperbolic polynomials is known to have nonempty interior
\cite{nuij1969note} and is of dimension $n^d$, whereas the set of $n-$tuples of
$d\times d$ matrices has dimension $n\binom{d}{2}$. This led to the weaker
conjecture that for every hyperbolic $p(x)$ there is an integer $k$ such that
$p(x)^k$ admits a definite determinantal representation, which was disproved by Br\"anden in \cite{branden2011obstructions}.
The Generalized Lax conjecture, which is a geometric statement, is equivalent to the yet weaker algebraic statement
that for every hyperbolic $p(x)$ there is a hyperbolic $q(x)$ such that $K_p\subseteq K_q$ and
$p(x)q(x)$ admits a definite determinantal representation.

For several special classes of hyperbolic polynomials, the corresponding
hyperbolicity cones are known to be spectrahedral.  The elementary
symmetric polynomial of degree $d$ in $n$ variables is given by $e_d(x) =
\sum_{S \subseteq [n], |S| = d} \prod_{i \in S} x_i$. Br\"anden \cite{branden2014hyperbolicity} 
 showed that the hyperbolicity cones of
elementary symmetric polynomials are spectrahedral with matrices of dimension $O(n^d)$.  If a polynomial $p$ is
hyperbolic with respect to a direction $e \in \R^n$, then its directional
derivatives along $e$ are hyperbolic too \cite{gaarding1959inequality}.
Directional derivatives of the polynomial $p(x) = \prod_{i \in [n]} x_i$
\cite{zinchenko2008hyperbolicity,sanyal2013derivative,branden2014hyperbolicity,choe2004homogeneous}
and the first derivatives of the determinant polynomial \cite{saunderson2017spectrahedral} are also known to
satisfy the generalized Lax conjecture. Amini \cite{amini2016spectrahedrality} has shown that 
certain multivariate matching polynomials are hyperbolic and that their cones are spectrahedral, of
dimension $n!$. 

Given the above examples, it is natural to wonder whether exponential
blowup in dimension is an essential feature of passing from hyperbolicity cones
to spectrahedral representations, and even assuming the generalized Lax conjecture to be
true, the size of the spectrahedral representation of a hyperbolicity cone is
interesting from a complexity standpoint. 
In this paper, we obtain exponential lower bounds on the size of the spectrahedral representation in general, even if
the representation is allowed to be approximate (up to an exponentially small error).

Recall that the Hausdroff distance between two cones
		$K$ and $K'$ is defined as $$\hdist(K,K') = \max_{x \in
		\mathcal{B} \cap K, y \in \mathcal{B} \cap K'} (d(x,K'),
		d(y,K)),$$ where $\mathcal{B}$ is the unit ball in $\R^n$.
We say that a spectrahedral cone $K'$ is  an {\em $\eta$-approximate spectrahedral representation} of another cone $K$ if
$\hdist(K,K')\le \eta$.
Our main theorem is the following.
\begin{theorem} [Main Theorem]\label{thm:main}
There exists an absolute constant $\kappa > 0$ such that for all sufficiently large $n, d$, 
there exists an n-variate degree $d$ hyperbolic polynomial $p$ whose hyperbolicity cone $K_p$ does not admit an $\eta-$approximate spectrahedral representation of dimension $B \leq \left(n/ d\right)^{\kappa d}$, for $\eta=1/n^{4nd}$. 
\end{theorem}

Our proof is analytic and does not rely on algebraic obstacles to
representability; in fact the polynomials we construct have very simple
coefficients (they are essentially binary perturbations of $e_d$).  However,
since there are $\binom{n}{d}$ coefficients, the examples require $\Omega(n^d)$
bits to write down. It is still unknown whether $e_d$ itself admits a low dimensional spectrahedral
representation, though Saunderson and Parrilo \cite{saunderson2015polynomial} have shown that if one allows
{\em projections} of sections of semidefinite cones, then such a represntation 
exists with size $\mathrm{poly}(n,d)$. 

Algebraically, the notion of (exact) spectrahedral representation of a
hyperbolicity cone $K_p$ for a polynomial $p$ corresponds to an algebraic
identity of the form, \[ p (x) \cdot q(x) = det ( \sum_{i \in [n]} C_i x_i) \ ,
\] where $K_q$ contains $K_p$.  Thus, our main theorem implies the existence of
a degree $d$ polynomial $p$ such that the degree of any identity of the above
form is at least $(n/d)^{\kappa d}$.

\subsection{Proof Overview}

The starting point for our proof is the theorem of Nuij \cite{nuij1969note},
which says that the space of hyperbolic polynomials of degree $d$ in $n$
variables has a nonempty interior, immediately implying that it has dimension
$n^d$. The Generalized Lax Conjecture concerns the {\em cones} of these
polynomials, which are geometric rather than algebraic objects. If we could
show that this space of cones also has large ``dimension'' in some appropriate quantitative
sense, and that the maps between the hyperbolicity cones and their spectrahedral representations are 
suitably well-behaved, then it would rule out the existence of small spectrahedral
representations for all of them uniformly, since such representations are parameterized
by tuples of $n$  $B\times B$ matrices, which have dimension $nB^2$. 

The difficulties in turning this idea into a proof are: (1) There are no
quantitative bounds on Nuij's theorem. (2) the space of hyperbolicity cones is
hard to parameterize and it is not clear how to define dimension. (3) the
mapping from hyperbolicity cones to their representations can be arbitrary, and
needn't preserve the usual notions of dimension anyway. We surmount these
difficulties by a packing argument, which consists of the following steps.

\begin{figure}
	\centering
	\includegraphics[width=0.8\textwidth]{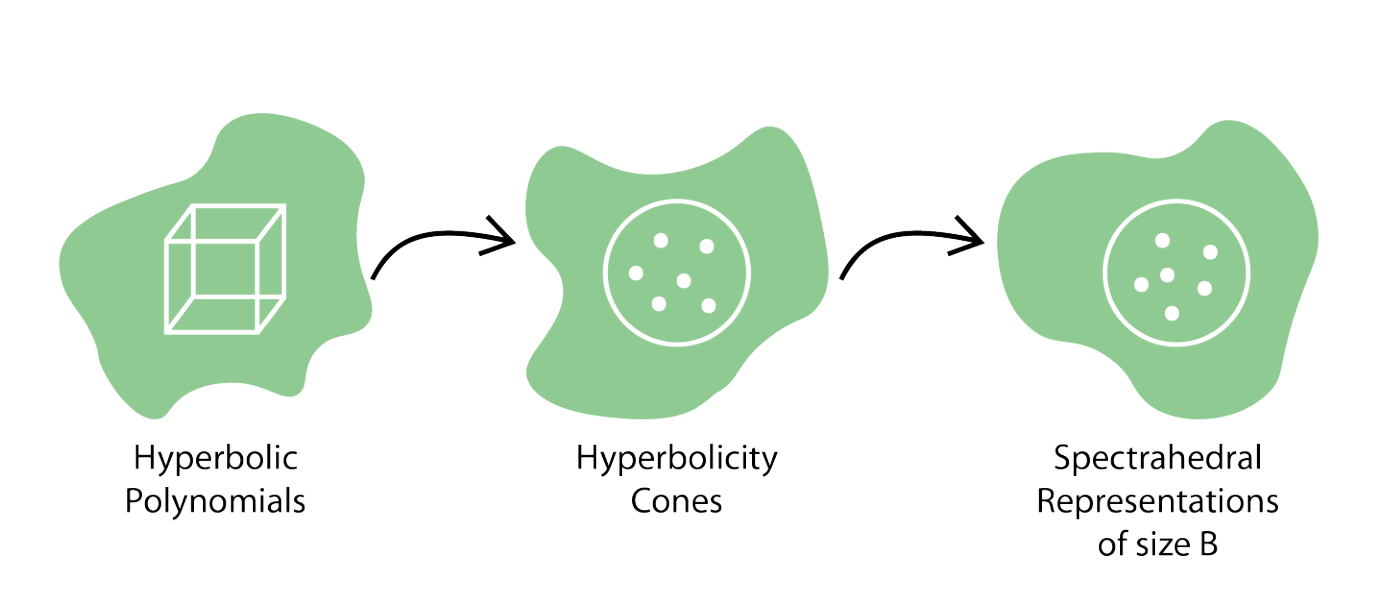}
	\caption{Outline of the Proof}
\end{figure}

\begin{enumerate}
	\item Exhibit a large family of hyperbolicity cones of size
		$2^{(n/d)^{\Omega(d)}}$, every pair of which are at least
		$\epsilon$ apart from each other in the Hausdroff metric
		between the cones.  

\item  Show that the spectrahedral representations of two distant cones $K$ and
	$K'$ are distant from each other, once the representations are
		appropriately normalized.  Formally, the matrices $\{C_i\}_{i
		\in [n]}$ and $\{C'_i\}_{i \in [n]}$ representing the two cones
		are at least $\epsilon'$ away in operator norm if $\hdist(K,K')
		\geq \epsilon$ (see Lemma \ref{lem:conetomatrices}). 

		We work only with cones which contain the positive orthant in order to ensure
		that they have normalized representations.

\item By considering the volume, there is a $(B/\epsilon')^{nB^2}$ upper bound
	on the number of pairwise $\epsilon'$-distant spectrahedral
		representations in $B\times B$ matrices, thus giving the lower
		bound on $B$.

\end{enumerate}

By far, the most technical part of the proof is the first step exhibiting a large family of pairwise distant hyperbolicity cones.  
The set of hyperbolic polynomials is known to have a non-empty interior in the
set of all degree $d$ homogenous polynomials.  Although this implies the
existence of a full-dimensional family of hyperbolic polynomials, it is not
clear how far apart they are quantitatively. Moreover, without any further understanding
of the structure of the polynomials, it is difficult to argue that their cones are far
in the Hausdorff metric.
To this end, we work with an explicit family of hyperbolic polynomials which
are all perturbations of the elementary symmetric polynomials, whose cones we
are able to understand.  Specifically, we will show the following:

\begin{enumerate}
\item There exists an explicit family $\mathcal{P}$ of $2^{\Omega(\binom{n}{d}
	)}$ perturbations of the degree $d$ elementary symmetric polynomial $e_d(x) = \sum_{S, |S| \leq d} \prod_{i \in S} x_i$ that are all hyperbolic, and pairwise distant from each other. These perturabations are indexed by a hypercube of dimension $\Omega(\binom{n}{d})$, as depicted in Figure 1.
The subspace of perturbations is carefully chosen to preserve real-rootedness
		of all the restrictions, thereby preserving the hyperbolicity
		of the polynomial $e_d(x)$ (see Section \ref{sec:many-ed}), as
		well as to yield perturbations of an especially simple
		structure.

\item The hyperbolicity cones for every pair of polynomials in $\mathcal{P}$ are far from each other.  In order to lower bound the Hausdroff distance between these cones, we identify an explicit set of $\Omega(\binom{n}{d})$ points on the boundary of the hyperbolicity cone for $e_d(x)$ as markers.  We will lower bound the perturbation of these markers as the polynomial is perturbed, in order to argue that the corresponding hyperbolicity cone is also perturbed (see Section \ref{sec:conesfar}). Again, the structure of the chosen subspace ensures that there are no ``interactions'' between the markers, and the analysis is reduced to understanding the perturbation of a single univariate Jacobi polynomial.
\end{enumerate}

\section{Many Hyperbolic Perturbations of $e_d$} \label{sec:many-ed}
In this section we prove that even though $e_d(x_1,\ldots,x_n)$ is not in the
interior of the set of Hyperbolic polynomials, there is a large subspace
$\Pi_d$ of homogeneous polynomials of degree $d$ in $n$ variables such
that all sufficiently small perturbations of $e_d$ in this subspace remain
hyperbolic. 

The subspace will be spanned by certain homogeneous polynomials corresponding to matchings. For any matching $M$ containing $d$ edges on $\{1,\ldots,n\}$, define the polynomial
$$q_M(x_1,\ldots,x_n):=\prod_{i<j\in M}(x_i-x_j).$$
We say that a $d-$matching on $[n]$ {\em fully crosses} a $d-$subset $S$ of $[n]$ if every edge of $M$ has exactly one endpoint in $S$.
\renewcommand{\S}{\mathcal{S}}
\begin{lemma}[Many Uniquely Crossing Matchings] \label{lem:manymatchings} There is a set $\M_d$ of $d-$matchings on $[n]$ of size at least 
$$|\M_d|=:N\ge \frac{\binom{n}{d}}{4\cdot 2^d}$$
and a set of $d-$subsets $\S_d$ such that for every $S\in\S_d$ there is a {\em unique} matching $M\in\M_d$ which fully crosses it, and for every $M\in\M_d$ there is at least one $S\in\S_d$ which it fully crosses.

Moreover, for every indicator vector $\one_S$, $S\in \S_d$ there is a unique $M\in\M_d$ such that $q_M(\one_S)\neq 0$, so the dimension of the span of 
$$\Pi_d:=\{q_M:M\in\M_d\}$$ is exactly $N$.
\end{lemma}
\begin{proof} Let $\M$ denote the set of all $d-$matchings on $[n]$ and let $\S$ be the set of all $d-$subsets of $[n]$. Let 
$$E:=\binom{n-d}{d}\cdot d!$$
be the number of matchings fully crossing a fixed set $S\in\S$. Let $\M_d$ be a random subset of $\M$ in which each matching is included independently with probability $\alpha/E$ for some $\alpha\in (0,1)$ to be determined later, and let $X_M$ be the indicator random variable that $M\in\M_d$. For any set $S$, define the random variable
\renewcommand{\deg}{\mathrm{deg}}
\newcommand{\E}{\mathbb{E}}
\renewcommand{\P}{\mathbb{P}}
$$\deg(S): = \sum_{M\textrm{ fully crossing }S} X_M,$$
and observe that
$$\E \deg(S) = \alpha.$$
Call a set $S$ {\em good} if $\deg(S)=1$ and let $G$ be the number of good $S$. Observe that
$$\E G = \binom{n}{d}\cdot\P[\deg(\{1,\ldots,d\}=1] = \binom{n}{d}\cdot E\cdot (1-(\alpha/E))^{E-1}\cdot\frac{\alpha}{E}\ge \binom{n}{d}\alpha(1-\alpha\cdot(1-1/E)).$$ 
Setting $\alpha=1/2$ we therefore have
$$\E G \ge \binom{n}{d}/4,$$
so with nonzero probability there are at least $\binom{n}{d}/4$
good subsets.

Let $\S_d$ be the set of good subsets. Finally, remove from $\M_d$ all the matchings that do not fully cross a good subset. Since there are at least $(1/4)\binom{n}{d}$ good subsets, every good subset fully crosses at least one matching in $\M_d$, and at most $2^d$ subsets cross any given matching, the number of matchings whichremain in $\M_d$ is at least $$(1/4)\binom{n}{d}/2^d,$$
as desired.

The moreover part is seen by observing that $q_M(\one_S)\neq 0$ if and only if $M$ fully crosses $S$.
\end{proof}
\begin{remark} In fact, the dimension of the span of all of the polynomials $q_M$ is exactly $\binom{n}{d}-\binom{n}{d-1}$, by properties of the Johnson Scheme, but in order to obtain the additional property above we have restricted to $\M_d$.\end{remark}

Henceforth we fix $\{q_M:M\in\M_d\}$ to be a basis of $\Pi_d$ and define the
$\ell_1$ norm of a polynomial $$q = \sum_{M\in\M_d} s_M q_M\in \Pi_d$$ as
$\|q\|_1:=\|s\|_1$. We will use the notation
$$m_d(x_1,\ldots,x_n)=\max_{|S|=d}\prod_{i\in S}|x_i|,$$ to denote the product
of the $d$ largest entries of a vector in absolute value, and occasionally we
will write $e_d(p)$ and $m_d(p)$ for a real-rooted polynomial $p$, which means applying
$e_d$ or $m_d$ to its roots. The operator $D$ refers to differentiation with respect to $t$.

The main theorem of this section is:
 \begin{theorem}[Effective Relative Nuij Theorem for $e_d$]\label{thm:perturb} If $q\in\Pi_d$ satisfies $$\|q\|_1\le
	 \frac{\binom{n}{d}}{2^n\cdot n^{(d+1)(n-d)}}=:R$$
	 then $e_d+q$ is hyperbolic with respect
 to $\one$.\end{theorem}
Recalling the definition of hyperbolicity, our task is to show that all of the
restrictions $$t\mapsto e_d(t\one+x),\quad x\in\R^n$$ remain real-rooted after perturbation by $q$.
Many of these restrictions lie on the boundary of the set of (univariate) real-rooted polynomials, 
 or arbitrarily close to it, so it is not possible to simply discretize the set of $x$ by a net
 and choose $q$ to be uniformly small on this net; one must instead carry out a more delicate restriction-specific analysis
 which shows that for small $R$, the perturbation $q(t\one +x)$ is less than the distance of $e_d(t\one + x)$
 to the boundary of the set of real-rooted polynomials, for each $x\in\R^n$.
Since we are comparing vanishingly small quantities, it is not a priori 
  clear that such an approach will yield an effective bound on $R$ depending only
  on $n$ and $d$; Lemma \ref{lem:edmd} shows that this is indeed possible.
\begin{proof}[Proof of Theorem \ref{thm:perturb}]
 Fix any nonzero vector $x\in\R^n$ and perturbation $q\in\Pi_d$ with $\|q\|_1\le R$ and consider the perturbed restriction:
 $$ r(t):= e_d(t\one+x)+q(t\one+x).$$
Let $p(t):=e_d(t\one+x)$ and observe that since $q$ is translation invariant, we have $q(t\one+x)=q(x)$, so in fact
$$r(t)=p(t)+q(x).$$
Let $\gamma\ge 0$ be the largest constant such that $p(t)+\delta$ is real-rooted for all $\delta\in[-\gamma,\gamma]$ (note that $\gamma$ could be zero if $p$ has a repeated root). It is sufficient to show that $|q(x)|\le \gamma$. Observe that 
\begin{equation}\label{eqn:gap}\gamma = \min_{t:p'(t)=0}|p(t)|,\end{equation}
since the boundary of the set of real-rooted polynomials consists of polynomials with repeated roots, and at any double root $t_0$ of the $p+\gamma$ we have $p'(t_0)=(p+\gamma)'(t_0)=(p+\gamma)(t_0)=0$. Let $t_0$ be the minimizer in \eqref{eqn:gap} and replace $x$ by $x-t_0\one$, noting that this translates $r(t)$ to $r(t-t_0)$ and does not change $\gamma$ or $q(x)$, so that we now have:
$$p'(0)=d\cdot e_{d-1}(x)=0$$
and
$$\gamma = |p(0)|=|e_d(x)|.$$

On the other hand, observe that:
\begin{align*}
|q(x)| &\le\sum_{M\in\M_d} |s_M|\prod_{ij\in M}|x_i-x_j|
\\&\le \|q\|_1\cdot\max_{M\in\M_d}\prod_{ij\in M}(|x_i|+|x_j|)
\\&\le \|q\|_1\cdot 2^d\cdot m_d(x).
\end{align*}
Thus we have $\gamma\ge |q(x)|$ as long as $m_d(x)=0$ or
$$ \|q\|_1\le \frac{|e_d(x)|}{2^dm_d(x)},$$
which is implied by
	$$ \|q\|_1\le \frac{\binom{n}{d}}{2^d(2n^{d+1})^{(n-d)}}$$
by Lemma \ref{lem:edmd}, as advertised.
\end{proof}

The following lemma may be seen as a quantitative version of the fact that if a real-rooted polynomial
has two consecutive zero coefficients $e_d=e_{d-1}=0$ then it must have a root of multiplicity $d+1$ at zero.
\begin{lemma}\label{lem:edmd} If $x\in\R^n$ satisfies $e_{d-1}(x)=0$ then
	$$ |e_d(x)|\ge \frac{\binom{n}{d}}{(2n^{d+1})^{(n-d)}}|m_d(x)|.$$
\end{lemma}
\begin{proof}

Let $p(t):=\prod_{i=1}^n(t-x_i)$ and let $q_k(t):=D^{n-k}p$, noting that $q_k$
is real-rooted of degree exactly $k$. Assume for the moment that all of the
$x_i$ are distinct and that $q_k(0)\neq 0$ for all $k=n,\ldots,d$ (note
that this is equivalent to assuming that the last $d+1$ coefficients of
$p$ are nonzero). Note that these conditions imply that all of the polynomials $q_k$ have distinct
roots, since differentiation cannot increase the multiplicity of a root.\\

If $n=d$ then the claim is trivially true since 
	\begin{equation}\label{eqn:base} e_d(x)=e_d(q_d)=m_d(q_d)=m_d(x).\end{equation}
Observe that $e_d$ behaves predictably under differentiation:
	\begin{equation}\label{eqn:edd} e_d(q_d)=e_d(D^{n-d}p)=\frac{(n-d)!}{n\cdot\ldots\cdot (d+1)}e_d(p)=\binom{n}{d}^{-1}e_d(q_n).\end{equation}
We will show by induction that:
	$$ m_d(q_{d})\ge \frac{1}{2n^{d+1}}m_d(q_{d+1})\ge\ldots \ge \frac{1}{(2n^{d+1})^{k-d}}m_d(q_k)\ge\ldots\ge \frac{1}{(2n^{d+1})^{n-d}}m_d(q_{n}),$$
	which combined with \eqref{eqn:base} and \eqref{eqn:edd} yields the desired conclusion.\\

\noindent {\bf Case $\mathbf{k=d+1}.$}  Let $z_-$ and $z_+$ be the smallest
	(in magnitude) negative and positive roots of $q_d=Dq_{d+1}$, respectively. 
	Let $w\neq 0$ be the unique root of $q_{d+1}$ between $z_-$ and $z_+$; assume without loss of generality that $w>0$ (otherwise 
	consider the polynomial $p(-x)$). Let $x_-< z_-$ and $x_+< z_+$ be the smallest
	in magnitude negative and positive roots of $q_{d+1}$ other than $w$, so that $x_- < 0<  w< x_+$. There
	are two subcases, depending on whether or not $w$ is close to zero --- if it is, then it prevents any root from
	shrinking too much under differentiation, and if it is not, the hypothesis $e_{d-1}(p)=0$ shows that $|z_-|$ and $|z_+|$ are
	comparable, which also yields the conclusion.
	\begin{itemize}
		\item Subcase $|w|\le |x_-|/2n$. By Lemma \ref{lem:aspect}, we have
			$$|z_-|\ge |x_-| - (|x_-|+|w|)(1-1/n) \ge |x_-|(1-(1+1/2n)(1-1/n))\ge |x_-|/2n.$$
			For every root of $q_{d+1}$ other than $x_-$ there is another root of $q_{d+1}$ between it and zero,
			so Lemma \ref{lem:aspect} implies that for every such root the neighboring (towards zero) root of $Dq_{d+1}$
			is smaller by at most $1/n$. Thus, we conclude that
			$$m_d(q_d)=m_d(Dq_{d+1})\ge \frac{m_d(q_{d+1})}{2n^{d}}.$$
		\item Subcase $|w|>|x_-|/2n$. In this case we may assume that $m_d(q_{d+1})$ is witnessed by the $d$ roots of $q_{d+1}$
			excluding $x_-$, call this set $W$, losing a factor of at most $1/2n$. Observe that every root in $W\setminus \{w\}$ 
			is separated from zero by another root of $q_{d+1}$, so such roots shrink by at most $1/n$ under differentiation
			by Lemma \ref{lem:aspect}. Noting that $x_-\notin W$, we have by interlacing that:
			$$ \prod_{q_d(z)=0, z\neq z_-}|z| \ge \frac{1}{n^{d-1}}\prod_{x\in W\setminus \{w\}} |x|,$$
			and our task is reduced to showing $|z_-|$ is not small compared to $|w|$

			The hypothesis $e_{d-1}(q_{d+1})=0$ implies that
		$q_d'(0)=0$; applying Lemma \ref{lem:aspect}, we find that the magnitudes of the innermost roots of $q_d$ must be comparable:
		$$|z_-|\lor |z_+|\le d\cdot (|z_-|\land|z_+|).$$ We now have
			$$|z_-|>|z_+|/n>|w|/n,$$
			so we conclude that
			$$m_d(q_d)\ge \frac{m_d(q_{d+1})}{2n^d},$$
			as desired.
	\end{itemize}

\noindent {\bf Case $\mathbf{k\ge d+2}$.} We proceed by induction.
Assume $m_d(q_k)$ is witnessed by a set of $d$ roots $L\cup R$, where $L$ contains negative roots
and $R$ contains positive ones. If there is a negative root not in $L$ and a positive root not in $R$
then as before every root in $L\cup R$ is separated from zero by another root of $q_k$, and by Lemma 7 we have
\begin{equation}\label{eqn:safe} m_d(Dq_k)\ge \frac{1}{n^d}m_d(q_k),\end{equation}
so we are are done. So assume all of the negative roots are contained in $L$; since $|L\cup R|=d$ this implies
that there are at least two positive roots not in $R$; let $z_*$ be the largest positive root not contained in $R$.
Let $z_-$ and $z_+$ be the negative and positive roots of $q_k$ of least magnitude. 
There are two cases:
\begin{itemize}
\item $|z_+|>|z_-|/2n$. This means that we can delete $z_-$ from $L$ and add $z_*$ to $R$, and reduce to the previous situation,
incurring a loss of at most $1/2n$, which means by \eqref{eqn:safe}:
$$ m_d(q_{k-1})\ge \frac{1}{2n}\frac{1}{n^d}m_d(q_k).$$
\item $|z_+|\le |z_-|/2n$. By  Lemma 7, the smallest in magnitude negative root of $Dq_k$ has magnitude at least
$$ (1-(1-1/n)(1+1/2n))|z_-|\ge |z_-|/2n,$$
and all the positive roots decrease by at most $1/n$ upon differentiating by Lemma 7, whence we have
$$m_d(Dq_k)\ge \frac{1}{2n^d}m_d(q_k).$$
\end{itemize}

To finish the proof, the requirements that $q_k(0)\neq 0$ for all $k$ and that
all coordinates of $x$ are distinct may be removed by a density argument, since
the set of $x$ for which this is true is dense in the set of $x\in\R^n$
satisfying $e_{d-1}(x)=0$.

\end{proof}
\begin{remark} We suspect that the dependence on $n$ and $d$ in the above lemma can be improved,
	and it is even plausible that it holds with a polynomial rather than exponential dependence of $R$ on $n$. 
	Since we do not know how to do this at the moment, we have chosen to present the simplest proof we know,
	without trying to optimize the parameters.\end{remark}

The following lemma is a quantitative version of the fact that the roots of the derivative of a polynomial
interlace its roots.
\begin{lemma}[Quantitative Interlacing] \label{lem:aspect} If $p$ is real rooted of degree $n$ then
		every root of $p'$ between two distinct consecutive roots of
		$p$ divides the line segment between them in at most the ratio
		$1:n$. \end{lemma}
  \begin{proof} Begin by recalling that 
		if $p$ has distinct roots $z_1< \ldots < z_n$ then the
		roots $z_1' < \ldots < z_{n-1}'$ of $p'$ satisfy
		for $j=1,\ldots,n-1$:
		$$ \sum_{i\le j}\frac1{z_j'-z_i} = \sum_{i>j}\frac1{z_i-z_j'}.$$
		Note that the solution $z_j'$ is monotone increasing in the
		$z_i$ on the LHS and monotone decreasing in the $z_i$ on the
		RHS. Thus, $z_j'$ is at least the solution to:
		$$ \frac1{z_j'-z_j}=\frac{n}{z_{j+1}-z_j'},$$
		which means that it is at least $z_j+\frac{z_{j+1}-z_j}{n}$. A
		similar argument shows that it it at most
		$z_j+(1-1/n)(z_{j+1}-z_j)$. Adding the common roots of $p$ and
		$p'$ back in, we conclude that these inequalities are satisfied
		by all of the $z_j'$.\end{proof}

\newcommand{\rdist}{\mathsf{rdist}}
\newcommand{\lmax}{\lambda_{max}}
\section{Separation in Restriction Distance}
For a parameter $\epsilon>0$ to be chosen later and $s\in\{0,1\}^{\M_d}$ let
$$p_s(x_1,\ldots,x_n):=e_d(x_1,\ldots,x_n)-\epsilon\sum_{M\in\M_d} s_M q_M(x_1,\ldots,x_n).$$
For any polynomial $p$ hyperbolic with respect to $\one$, define the {\em restriction embedding}
$$ \Lambda(p):=(\lmax(p(t\one+\one_S)))_{S\in\S_d}.$$

\begin{lemma} \label{lem:restrictfar} If $0<\epsilon<R/N$ and $s,s'\in\{0,1\}^{\M_d}$ are distinct then both $p_s$ and $p_s'$ are hyperbolic
with respect to $\one$ and
$$\|\Lambda(p_s)-\Lambda(p_{s'})\|_\infty\ge \Delta:=\epsilon \left(\binom{n}{d} d e\right)^{-1}.$$
\end{lemma}

\begin{proof} Since $\epsilon<R/N$, we have $\|e_d-p_s\|_1,\|e_d-p_{s'}\|_1 \le R$ so by Theorem \ref{thm:perturb} both of them must be hyperbolic with respect to $\one$.

Since $s\neq s'$, suppose $s_M=0$ and $s'_M=1$ for some matching $M\in\M_d$, and let $S\in \S_d$ be a set which fully crosses $M$.  By Lemma \ref{lem:manymatchings}, $M$ is the {\em only} matching in $\M_d$ which crosses $S$, so we have $q_M(\one_S)=1$ and $q_{M'}(\one_S)=0$ for all other $M'\neq M\in\M_d$. Thus, along the restriction $t\one +\one_S$, one has
\begin{equation}\label{eqn:jacobi} q_s(t\one+\one_S) = e_d(t\one + \one_S)=:J(t)\end{equation}
and
$$q_{s'}(t\one+\one_S) = e_d(t\one + \one_S)-\epsilon q_M(\one_S) = e_d(t\one+\one_S)-\epsilon=J(t)-\epsilon,$$
where we have again used that the $q_M$ are translation invariant. Note that $J(t)$ has positive leading coefficient, so subtracting a constant from it increases its largest root. Thus, our task is reduced to showing that
$$\lambda_{max}(J(t)-\epsilon)\ge \lambda_{max}(J(t))+\Delta.$$

To analyze the behavior of this perturbation, first we note that 
$$J(t)= e_d(t1, \ldots, t1, t1 + 1, \ldots, t1 + 1) = \frac{1}{(n-d)!} D^{n-d} e_N(t \one + \one_S) = \frac{1}{(n-d)!} D^{n-d}t^{n-d} (t+1)^d.$$
Since $t^{n-d}(t+1)^d$ has roots in $[-1,0]$ and the roots of the derivative of a polynomial interlace its roots, 
we immediately conclude that the zeros of $J$ satisfy $-1\le z_1\le \ldots \le z_d\le 0$. 
Again by interlacing, we see that $J'(z)\ge 0$ and $J''(z)\ge 0$ for $z\ge z_d$, whence $J$ is monotone and convex above $z_d$. Thus, $\lambda:=\lambda_{max}(J(t)-\epsilon)$ is the least $\lambda>z_d$ such that $J(\lambda)\ge \epsilon$. Let $\theta>0$ be a parameter to be set later. Then either $\lambda\ge z_d+\theta$ or we have by convexity that
$$J(z_d+\theta)\le J(z_d)+\theta J'(z_d+\theta).$$
The first term is zero and we can upperbound the second term as:
$$|J'(z_d+\theta)| \le \binom{n}{d} \sum_{i\le d}|\prod_{j\neq i}(z_d+\theta-z_i)|\le \binom{n}{d} d\cdot (1+\theta)^{n-1},$$
since $|z_d-z_i|\le 1$ for every $i\le d$. Thus, we have
$$ J(z_d+\theta)\le \binom{n}{d}d\theta(1+\theta)^{n-1}\le \binom{n}{d} d\theta e$$
whenever $\theta<1/n$, which is less than $\epsilon$ for $\theta=\epsilon (\binom{n}{d} d e)^{-1}$. This means that in either case we must have
$$\lambda\ge z_d+	\epsilon \left(\binom{n}{d} d e\right)^{-1},$$
as desired.
\end{proof}

\section{Separation in Hausdorff Distance} \label{sec:conesfar}
\begin{lemma} \label{lem:restrictiontocone} For $s, s'\in\{0,1\}^{\M_d}$ and 
$$\epsilon<\frac{1}{4n^{d(n-d)}Nd\sqrt{n}}=:R_2$$
we have
 $$\hdist(K_{p_s},K_{p_{s'}})>\frac{\|\Lambda(p_s)-\Lambda(p_s')\|_\infty}{18n^{d(n-d)}Nn}$$
\end{lemma}

\begin{proof}
Suppose $\|\Lambda(p_s)-\Lambda(p_s')\|_\infty=\Delta>0$. There must be some restriction $t \one + \one_S$ along which the intersections with the boundaries of the cones differ by $\Delta$. Moreover, this $S$ must be fully crossed by some matching $M$ for which $s_M$ or $s_{M'}$ is nonzero, since otherwise we would have $p_s(t\one + \one_S)=p_{s'}(t\one +\one_S)=e_d(t\one+\one_S)$. By Lemma \ref{lem:manymatchings}, there is a unique such $M$ and we may assume that $s_M=0$ and $s_M'=1$, and moreover $q_{M'}(t\one+\one_S)=0$ for all other $M'\in\M_d$, so that we have
$$p_s(\lambda\one+\one_S)=e_d(\lambda\one+\one_S)=0$$
and
$$p_{s'}((\lambda+\Delta)\one+\one_S) = e_d((\lambda+\Delta)\one+\one_S)-\epsilon=0$$
where $\lambda\in\R$ is the largest root of the polynomial $J(t)$ from \eqref{eqn:jacobi}.

Let 
$$z=\lambda\one+\one_S$$
and 
$$z'=(\lambda+\Delta)\one+\one_S=z+\Delta\one$$ denote the corresponding points in $\R^n$ on the boundaries of $K_{p_s}$ and $K_{p_{s'}}$, respectively.
Let $H$ be the hyperplane tangent to $K_{p_s}$ at $z$. Let $v$ be a unit vector normal to $H$. Since the hyperbolicity cones are convex, the distance from $z'$ to $K_{p_s}$ is at least the distance to $H$, which is given by:
\[ \|\Delta \one\|\cdot \frac{|\langle \one, v \rangle |}{\| \one \| \| v\|}=\Delta\cdot\frac{|\langle\one,v\rangle|}{\|v\|}. \]
Normalizing so that $z'$ is a unit vector, we obtain:
\begin{equation} \label{eqn:zwv} \hdist(K_{p_{s}},K_{p_{s'}}) \geq \Delta \frac{|\langle \one, v \rangle|}{\| v\| \|z'\|},\end{equation}
so if we can prove a uniform lower bound on this quantity over all $v$ and $z'$ corresponding to $s,s'$, then we are done.

Computing the normal we find that 
$$v = \nabla p_s (z) = \nabla e_d(z) + \epsilon \sum_{M\in\M_d} s_M \nabla q_M(z)=:\nabla e_d(z)+\epsilon w,$$
so that
$$ |\langle \one, v\rangle| \ge |\langle \one,\nabla e_d(z)\rangle|-\epsilon\|\one\|\|w\|.$$
The first inner product is just the directional derivative of $e_d(t\one+\one_S)$ along $\one$ at $t=\lambda$, which is:
\begin{equation} \label{eqn:jaclb}
\frac{d}{dt}e_d(t\one+\one_S)|_{t=\lambda}=J'(\lambda)\ge \frac{1}{n^{d(n-d)}},\end{equation}
by Lemma \ref{lem:jacobi}, proven below.

We now prove crude upper bounds on $\|z'\|,\|w\|, \|v\|$, which will be negligible when $\epsilon$ is small. First we have
\begin{equation}\label{eqn:zbound}\|z'\|\le (\lambda+\Delta)\|\one\|+\|\one_S\|\le 3\sqrt{n},\end{equation}
since $|\lambda|\le 1$ because $J$ has roots in $[-1,0]$ and $|\Delta|<1.$

To control $\|w\|$, we compute the $i$th coordinate of $\nabla q_M(x)$:
$$\partial_{x_i} q_M(x)=\sigma_i\prod_{k<l\in M\setminus\{i\}} (x_k-x_l),$$
where $\sigma_i$ zero if $i\notin M$ and $\pm 1$ if $i\in M$. Since $z$ has coordinates in $[-1,0]$, $|\partial_{x_i}q_M(z)|\le 1$ for all $i\in M$ and 
$$\|\nabla q_M(z)\|\le \sqrt{2d}.$$
Applying the triangle inequality and noting that $s_M\in \{0,1\}$ gives
\begin{equation}\label{eqn:wbound} \|w\|\le |\M_d|\cdot \sqrt{2d}\le 2Nd.\end{equation}

Finally, we have
$$\|\nabla e_d(z)\|=\sqrt{\sum_{i=1}^ne_{d-1}(z_{-i})^2},$$
where $z_{-i}$ is the vector obtained by deleting the $i^{th}$ coordinate of $z$. Since these coordinates are bounded in magnitude by $1$, the above norm is bounded by
$$ \sqrt{n}\cdot\binom{n}{d-1}\le \sqrt{n}\cdot N,$$
and applying the triangle inequality once more we get
\begin{equation} \label{eqn:vbound}\|v\|\le \sqrt{n}\cdot N + \epsilon\|w\|\le 3N\sqrt{n},\end{equation}
since $\epsilon<1$.

Combining \eqref{eqn:zwv}, \eqref{eqn:jaclb}, \eqref{eqn:zbound},\eqref{eqn:wbound}, and \eqref{eqn:vbound}, we have:
$$\hdist(K_{p_s},K_{p_{s'}})\ge \Delta\cdot \frac{n^{-d(n-d)}-\epsilon\cdot 2Nd\sqrt{n}}{3\sqrt{n}\cdot 3N\sqrt{n}}\ge \frac{\Delta}{18n^{d(n-d)}Nn},$$
provided
$$2Nd\sqrt{n}\epsilon \le (1/2)n^{-d(n-d)},$$
as desired.
\end{proof}

\begin{lemma}[Sensitivity of Jacobi Root]\label{lem:jacobi} Let $\lambda$ be the largest root of $J(t)$. Then
$$J'(\lambda)\ge \frac{1}{n^{d(n-d)}}$$.
\end{lemma}
The proof is deferred to the appendix.

\newcommand{\kernel}{\mathsf{Ker}}
\newcommand{\Id}{\mathsf{Id}}

\section{Separation of Matrix Parameterizations}

Given $C = \{C_i\}_{i \in [n]} \subset \R^{k \times k}$, define the cone 
\[ K_C = \{ x \in \R^n \mid \sum_{i \in [n]}  C_i x_i \succeq 0\}   \]
Here $C$ is said to be a spectrahedral representation of the cone $K_C$.

\begin{definition}
	A spectrahedral representation of a cone $K  \subseteq \R^n$ as $K = \{ x \in \R^n | \sum_i C_i x_i \succeq 0\}$,is said to be a {\it normalized} if,
    	\[ \sum_{i = 1}^n C_i = \Id_k \]
        and $C_i \succeq 0$
 \end{definition}
 
 \begin{lemma} \label{lem:normalization}
	If a spectrahedral cone $K$ contains the positive orthant $\R_{+}^n$ then $K$ admits a normalized representation.
 \end{lemma}

 \begin{proof}
Let $C = \{C_i\}_{i \in [n]}$ be a spectrahedral representation of $K$.
Let $U = \cap_{i = 1}^n \kernel(C_i)$ be the subspace in the kernel of all the $C_i$, and let $\Pi_{U^\perp}$ denote the projection on to $U^{\perp}$.  It is easy to check that for all $x \in \R^n$,
 \[ \sum_i C_i x_i \succeq 0 \iff \sum_i \Pi_{U^{\perp}} C_i \Pi_{U^{\perp}} x_i \succeq 0 .\]
 By a basis change of $\R^n$, that contains a basis for $U^{\perp}$, we obtain matrices $C' = \{C'_i\}_{i \in [n]}$ in $\R^{k' \times k'}$ where $k' = \dim(U^{\perp})$ such that $K = K_{C'}$.
 
Furthermore, since the $i^{th}$ basis vector $e_i$ is in the non-negative orthant, $e_i \in K$.  This implies that $C'_i \succeq 0$ for each $i$.  For each $u \in U^{\perp}$, there exists $C_i$ such that $u^T C_i u > 0$, which implies that 
$\sum_i u^T  C'_i u > 0$.  In other words, we have $M = \sum_i C'_i \succ 0$ is positive definite.

The normalized representation of $K$ is given by $C'' = \{ M^{-1/2}C'_i M^{-1/2}\}_{i \in [n]}$.  By definition, the representation is normalized in that $\sum_{i \in [n]} C''_i = \Id_{k'}$.
 \end{proof}
 
For two spectrahedral representations given by matrices $ C = \{C_i\}_{i \in [n]} $ and $C' = \{C'_i\}_{i \in [n]}$, define the distance between the representations as,
\[ \mdist(C, C')  = \max_i \| C_i - C'_i \| \]

\begin{lemma} \label{lem:conetomatrices}
Suppose $C, C'$ are normalized spectrahedral representations of $K_C$ and $K_{C'}$ respectively.  Then,
$\hdist(K_C, K_{C'}) \leq n^{3/2} \cdot \mdist(C,C')$
\end{lemma}
\begin{figure}
	\centering
	\includegraphics[width=0.6\textwidth]{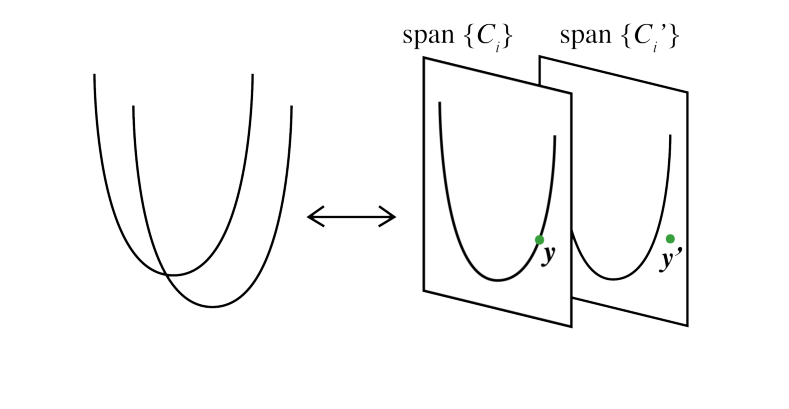}
	\caption{Proof of Lemma \ref{lem:conetomatrices}}
\end{figure}

\begin{proof}
Let $B(0,1)$ denote the $\ell_\infty$ unit ball in $\R^n$.  For every $x \in K_C \cap B(0,1)$, 
\begin{align*}
 \sum_{i \in [n]} C'_i x_i & = \sum_{i \in [n]} C_i x_i + \sum_{i \in [n]} (C_i-C'_i) x_i \\
 & \succeq - n \mdist(C,C') \cdot \Id_k 
 \end{align*}
where we are using the fact that $\sum_i C_i x_i \succeq 0$.  This implies that the point $x' =  x + n \mdist(C,C') \cdot \one \in K_{C'}$.  To see this, recall that the representation is normalized in that $\sum_i C'_i = \Id_k$.  Therefore, we get 
\[ \sum_i C'_i x'_i = \sum_i C'_i x_i + n \cdot \mdist(C,C') \sum_i C'_i \succeq 0 \]
The lemma follows by observing that $\|x - x'\|_2 \leq n^{3/2} \cdot \mdist(C,C')$.

\end{proof}

\section{Proof of The Main Theorem}
We restate the theorem for convenience.
\begin{theorem} [Restatement of Main Theorem]\label{thm:main}
There exists an absolute constant $\kappa > 0$ such that for all sufficiently large $n, d$, 
there exists an n-variate degree $d$ hyperbolic polynomial $p$ whose hyperbolicity cone $K_p$ does not admit an $\eta-$approximate spectrahedral representation of dimension $B \leq \left(n/ d\right)^{\kappa d}$, for $\eta=1/n^{4nd}$. 
\end{theorem}

\begin{proof} Set $\epsilon$ smaller than $R$ and $R_2$.
Theorem \ref{thm:perturb}, Lemma \ref{lem:restrictfar} and Lemma \ref{lem:restrictiontocone}
  together imply a family of hyperbolic polynomials $\mathcal{P}$ with the following properties:
\begin{enumerate}
	\item $|\mathcal{P}| = 2^{|\mathcal{M}_d|}$ where $|\mathcal{M}_d | \geq (n/d)^{\kappa d}$ for some absolute constant $\kappa > 0$.        
    
    \item For all $p \neq p' \in \mathcal{P}$, $\hdist(K_p,K_p') \geq \gamma$ for $\gamma > \frac{1}{n^{3nd}}$.
    
    \item For every $p \in \mathcal{P}$, the positive orthant $\R_+^n$ is contained in $K_p$.
    \end{enumerate}
The last observation follows from the fact that positive orthant is contained in $K_{e_d}$, and the perturbations of $e_d$ in $\mathcal{P}$ are small enough to keep all the coefficients non-negative.

Suppose each of the hyperbolicity cones $K_p$ admitted a $\gamma/3-$approximate spectrahedral representation in dimension $B$.  By Lemma \ref{lem:normalization}, for each cone in $K_p$, there exists a normalized $\gamma/3$-approximate spectrahedral representation $C_p$ in dimension $B_p \leq B$
	Further, by Lemma \ref{lem:conetomatrices}, for every pair of polynomials $p, p' \in \mathcal{P}$, their corresponding $\gamma/3$-approximate spectrahedral representations satisfy $\mdist(C_p, C_{p'}) \geq n^{-3/2} \gamma/3 \coloneqq \eta$.  

Notice that in every normalized spectrahedral representation $C_p$ in dimension $B$, every matrix $C \in C_p$ satisfies $C \succeq 0$ and $C \preceq \Id_B$.  This implies that $\|C\|_{F} \leq \sqrt{B}$.  By a simple volume argument, for every $\eta > 0$, the number of normalized  spectrahedral representations in $\R^{B\times B} $ whose pairwise distances are $\geq \eta$ is at most $\left( \sqrt{B} /\eta \right)^{nB^2}$.  Since every cone $K_p$ for $p \in \mathcal{P}$ admits a normalized spectrahedral representation of dimension at most $B$, we get that
\[ |\mathcal{P}| \leq t \left(\sqrt{B} /\eta \right)^{nB^2} \ ,\]
which implies that
\[ B^2 \log B \geq  |\mathcal{M}_d |/\log(n^{3/2}/ \gamma) \geq \frac{1}{n \log n}  \cdot (n/ d)^{\kappa d}  \ . \]
which implies the lower bound of  $B \geq (n/d)^{\kappa' d}$ for some constant $\kappa' > 0$.
\end{proof}  

\subsection*{Acknowledgments} We thank Jim Renegar for pointing out a mistake in a previous version of the proof of Lemma \ref{lem:edmd}. We thank
the Simons Institute for the Theory of Computing and MSRI, where this work was partially carried out, for their hospitality.

\providecommand{\bysame}{\leavevmode\hbox to3em{\hrulefill}\thinspace}
\providecommand{\MR}{\relax\ifhmode\unskip\space\fi MR }
\providecommand{\MRhref}[2]{%
  \href{http://www.ams.org/mathscinet-getitem?mr=#1}{#2}
}
\providecommand{\href}[2]{#2}

\appendix
\section*{Appendix}
\begin{proof}[Proof of Lemma \ref{lem:jacobi}]
Let $z_1\le,\ldots,\le z_d=\lambda$ be the roots of $J(t)$. Observe that
$$J'(\lambda) = \prod_{i<d}(z_d-z_i)\ge |z_d-z_{d-1}|^{d-1},$$
so if we can prove a lowerbound on the spacing between $z_d$ and $z_{d-1}$ we are done.

We do this by recalling from Lemma \ref{eqn:jacobi} that:
$$J(t)=\frac{1}{d!}D^{n-d}t^{n-d}(1+t)^{d}.$$
Let $q_k(t):=D^{k}t^{n-d}(1+t)^{d}$ and note that $q_{k+1}$ interlaces $q_k$, and all of these polynomials have real roots in $[-1,0]$. Let $y_k$ be the largest root of $q_k$ that is strictly less than $0$, noting that every $q_k$ for $k<{n-d}$ has a root at $0$ because $q_0$ has a root of multiplicity $n-d$ at $0$. By lemma \ref{lem:aspect}, we have for $k=1,\ldots,n-d$:
$$y_{k+1}\le 0-\frac{1}{n}|0-y_{k}|\le -|y_k|/n.$$
Noting that $y_0=-1$ and iterating this bound $n-d-1$ times we obtain
$$y_{n-d-1}\le -1/n^{n-d-1},$$
so that $D^{n-d-1}q_0$ has a root at zero and a root $y_{n-d-1}\le -1/n^{n-d}.$ Since $J(t)=Dq_{n-d-1}(t)$, we must have $z_{d-1}\le -y_{n-d-1}$ by interlacing. However, applying Lemma \ref{lem:aspect} once more, we see that
$$ z_d\ge -y_{n-d-1}+|0+y_{n-d-1}|/n = -(1-1/n)y_{n-d-1},$$
so the gap must be at least
$$z_d-z_{d-1}\ge y_{n-d-1}/n = \frac{1}{n^{n-d}}.$$
Thus, we conclude that
$$J'(\lambda)\ge \frac{1}{n^{(d-1)(n-d)}}\ge \frac{1}{n^{d(n-d)}},$$
as desired.
\end{proof}


\end{document}